\documentclass[11pt,reqno]{amsart} 

\usepackage{amscd}
\usepackage{amsfonts,amssymb,latexsym} 
\usepackage[all]{xy} 
\xyoption{arc}
\CompileMatrices

\newcommand{\udots}{\mathinner{\mskip1mu\raise1pt\vbox{\kern7pt\hbox{.}}
\mskip2mu\raise4pt\hbox{.}\mskip2mu\raise7pt\hbox{.}\mskip1mu}}

\newcommand{\SO}{{\mathcal{O}}}

\newcommand{\PP}{\mathbb{P}}

\newcommand{\Ext}{\operatorname{Ext}}

\newcommand{\Hilb}{\operatorname{Hilb}}
\newcommand{\Hom}{\operatorname{Hom}}

\newcommand{\too}{\longrightarrow}

\newcommand{\wt}{\widetilde}

\newcommand{\gon}{\operatorname{gon}}

\newcommand{\op}{\operatorname}

\newtheorem{proposition}{Proposition}[section]
\newtheorem{theorem}[proposition]{Theorem}

\newtheorem{lemma}[proposition]{Lemma}
\newtheorem{conjecture}[proposition]{Conjecture}
\newtheorem{corollary}[proposition]{Corollary}

\theoremstyle{remark}

\newtheorem{definition}[proposition]{Definition}
\newtheorem{remark}[proposition]{Remark}

\numberwithin{equation}{section}

\begin{document}

\title[Semi-orthogonal decomposition of symmetric products of curves]{Semi-orthogonal decomposition of symmetric 
products of curves and canonical system}

\author[I. Biswas]{Indranil Biswas}
\address{School of Mathematics, Tata Institute of Fundamental
Research, Homi Bhabha Road, Mumbai 400005, India}
\email{indranil@math.tifr.res.in}

\author[T. G\'omez]{Tom\'as L. G\'omez}
\address{Instituto de Ciencias Matem\'aticas (CSIC-UAM-UC3M-UCM),
Nicol\'as Cabrera 15, Campus Cantoblanco UAM, 28049 Madrid, Spain}
\email{tomas.gomez@icmat.es}

\author[K.-S. Lee]{Kyoung-Seog Lee}
\address{Center for Geometry and Physics, Institute for Basic Science
(IBS), Pohang 37673, Republic of Korea} 
\email{kyoungseog02@gmail.com}



\begin{abstract}
Let $C$ be an irreducible smooth complex projective curve of genus $g \,\geq\, 2$ and $C_d$ its $d$-fold 
symmetric product. In this paper, we study the question of semi-orthogonal decompositions 
of the derived category of $C_d$. This entails investigations of the canonical system
on $C_d$, in particular its base locus.
\end{abstract}

\maketitle

\section{Introduction}

Let $C$ be an irreducible smooth complex projective curve of genus $g \,\geq\, 2.$ Let $C^d\,=\,C\times \stackrel{d}\cdots 
\times C$ be its Cartesian product, and let $C_d$ be the $d$-fold symmetric product, meaning the quotient of 
$C^d$ by the action of the symmetric group of $d$ letters. We address the question whether the bounded derived category of 
coherent sheaves \[D(C_d)\,:=\,D^b_{coh}(C_d)\] admits a non-trivial semi-orthogonal decomposition (Definition \ref{defsod}).

Semi-orthogonal decomposition is one of the basic notions in the
theory of derived categories of coherent sheaves on algebraic
varieties. When the derived category of an algebraic variety admits a
semi-orthogonal decomposition, we can divide the category into smaller
pieces and try to understand the whole triangulated category via its
components. It turns out
that semi-orthogonal decompositions of derived categories of algebraic
varieties are closely related to birational geometry, Hochschild
homology and cohomology, K-theory, mirror symmetry, moduli theory,
motives, etc. See \cite{Ku14} and
\cite{AB} and references therein for an overview of the role of semi-orthogonal decompositions in algebraic geometry.
As a concrete example, we can mention the result of Bernardara and Bolognesi \cite{BB} giving a criterion for the rationality of a conic bundle on a minimal rational surface in terms of the existence of certain semi-orthogonal decomposition of its derived category. Also, there is a conjecture of Kuznetsov \cite[Conjecture 1.1]{Ku10} about the rationality of a smooth cubic fourfold in terms of one of the components of certain semi-orthogonal decomposition. Last but not least, semi-orthogonal decomposition can be used to investigate geometry of moduli spaces of instanton or ACM bundles on some Fano varieties \cite{Ku12}.

In general, it is very hard to classify all possible semi-orthogonal
decompositions of the derived category of a given variety. Even the
question of which algebraic variety can have a non-trivial
semi-orthogonal decomposition is still widely open. 

There are several known classes of varieties who do not admit non-trivial semi-orthogonal decomposition (cf. \cite{KO, Ok}). From the earlier works, we can see that the existence of non-trivial semi-orthogonal decomposition of the derived category of an algebraic variety is closely related to the base locus of the canonical bundle of the variety. This motivated us to study base locus of the canonical line bundle of $C_d$.

This study of the base locus of the canonical bundle and semi-orthogonal decomposition of the derived category of $C_d$ was inspired by a conjecture of M. S. Narasimhan. Recently, Narasimhan proved in \cite{Na1, Na2} that the derived category of $C$ can be embedded into the derived category of the moduli space $\mathcal{SU}_C(2,L)$ of stable bundles of rank 2 and fixed determinant $L$ of degree $1$. A similar result was obtained by Fonarev and Kuznetsov for general curves via different method (cf. \cite{FK}). Narasimhan conjectured that the derived category of the moduli space admits a semi-orthogonal decomposition as follows (Belmans, Galkin and Mukhopadhyay have independently stated the same conjecture in \cite{BGM}).

\begin{conjecture}\label{mainconjecture}
The derived category of $\mathcal{SU}_C(2,L)$ has the following semi-orthogonal decomposition
$$
D(\mathcal{SU}_C(2,L))=\langle D(pt), D(pt), D(C), D(C), \cdots,
D(C_{g-2}), D(C_{g-2}), D(C_{g-1}) \rangle.
$$
(two copies of $D(C_i)$ for $i\,<\,g-1$ and one copy for $i\,=\,g-1$).
\end{conjecture}

It turns out that there is a motivic decomposition of $\mathcal{SU}_C(2,L)$ which is compatible with the above 
conjecture (cf. \cite{Lee}). From the point of view of this conjecture it is of interest whether the 
derived categories of symmetric powers of curves can be further decomposed. Okawa proved that the derived 
category of a curve of genus $g \,\geq\, 2$ cannot have a non-trivial semi-orthogonal decomposition \cite{Ok}.
On the other hand, Toda gives an
explicit semi-orthogonal decomposition of $C_d$ for large $d$ in \cite[Corollary 5.11]{To}.
Let $J$ denotes the Jacobian of $C$. Toda proves that,
if $d\, > \,g-1$, then:
$$
D(C_d)=\langle \overbrace{D(J),{\ldots} ,D(J)}^{d-g+1},D(C_{-d+2g-2}) \rangle.
$$

Recall that the gonality $\gon(C)$ of a curve $C$ is the lowest degree among all nonconstant
morphisms from $C$ to the projective line $\PP^1$. Equivalently, it is
the lowest degree of a line bundle $L$ on $C$ with $h^0(L)\,\geq\,2$. 
In this paper, we prove the following theorem.

\begin{theorem}[Corollary \ref{cor:cd-sod}]
Let $C$ be a smooth complex projective curve of genus $g\,\geq\, 3$, and
let $d$ be a positive integer with $d\,<\,\gon(C)$. Then there is no non-trivial
semi-orthogonal decomposition of $D(C_d)$.
\end{theorem}

We note that, for a generic curve $C$ of genus $g$, the gonality satisfies
$$
\gon(C)\,= \, \left \lfloor{\frac{g+3}{2}}\right \rfloor. 
$$

If $d\,=\,2$ we prove the following result for all curves (no
condition on gonality) with genus at least $3$: 

\begin{theorem}[Theorem \ref{c2-sod}]
Let $C$ be a smooth projective curve of genus $g \,\geq \,3$.
Then there is no non-trivial semi-orthogonal decomposition on $D(C_2)$.
\end{theorem}

This result is sharp because, if $g\,=\,2$ then $D(C_2)$
admits a semi-orthogonal decomposition
(recall that, for $g\,=\,2$, the Albanese map $C_2\too J$ is the blow-up of the
Jacobian at a point, and then apply the semi-orthogonal decomposition
formula for the blow-up \cite{Or}).
We conjecture the following (this has independently been stated in \cite{BGM}).

\begin{conjecture}
\label{conjecture for symmetric power}
Let $C$ be a projective smooth curve of genus $g \geq 2.$ Then there is no non-trivial semi-orthogonal decomposition on $D(C_d)$ for $1 \leq d \leq g-1$.
\end{conjecture}

In this direction, we prove some results on the base locus of the canonical divisor $K_{C_d}$ of the symmetric product $C_d$.

\begin{proposition}[Proposition \ref{baselocus}]
Let $1\,\leq \,d\,\leq\, g-1$. The base locus of the canonical divisor $K_{C_d}$
is the set of points $(x_1,\,\cdots,\,x_d)$ in $C_d$ such that $h^0(\SO_C(x_1+\cdots+x_d))\,>\,1$.

Equivalently, the base locus is the set of points in $C_d$ where the Albanese map is not injective. 
\end{proposition}
 
The following conjecture is known to the experts. 

\begin{conjecture}
\label{conjecture}
Let $X$ be a smooth projective variety. If the canonical bundle $K_X$ is
nef and $h^0(K_X)\,>\,0$, then $X$ admits no non-trivial semi-orthogonal decomposition.
\end{conjecture}

In Lemma \ref{usingconj} we prove that Conjecture \ref{conjecture} implies Conjecture \ref{conjecture for symmetric power}

\section{Nefness of the canonical divisor of $C_d$}

Let $\Theta$ be the theta divisor on the Jacobian $J(C)$.
Fixing a point $p\in C$, the Albanese map of the symmetric
product $C_d$ is constructed as follows
$$
u\,:\,C_d \,\too\, J(C)\, ,\ \ D \,\longmapsto\, \SO_C(D-dp)\, .
$$
Note that the fiber of the Albanese map is
\begin{equation}\label{fiberalbanese}
u^{-1}(u(D))\,=\,\PP H^0(\SO_C(D))\, .
\end{equation}
We also define
$$
i\,:\,C_{d-1} \too C_{d}\, ,\ \ D\,\longmapsto\, D+p\, .
$$
Let $\theta\,:=\,u^* \Theta$. The class of the divisor
$i(C_{d-1})$ of $C_d$ will be denoted by $x$.

\begin{lemma}
\label{lemma:canonicald}
The canonical class of the symmetric product $C_d$ is given by the formula
\begin{equation}\label{eq:canonicald}
K_{C_d}\,=\, (g-d-1) x + \theta\, .
\end{equation}
\end{lemma}

\begin{proof}
Let $\Delta\, \subset\, C^d$ be the big diagonal where at least two points coincide. The image of
$\Delta$ under the quotient map $\pi\, :\, C^d\, \longrightarrow\, C_d$, for the action of the
symmetric group, will be denoted by $\Delta'$, so we have the commutative diagram
$$
\xymatrix{
{\Delta} \ar[r] \ar[d]_{\pi|_\Delta} & {C^d} \ar[d]^{\pi}\\
{\Delta'} \ar[r] & {C_d}
}
$$
Note that $\pi^*(\Delta')\,=\,2\Delta$, and hence $\Delta$ is
the ramification divisor.
The divisor $\Delta'$ is divisible by $2$; in fact,
$$
K_{C_d}\,=\, (2g-2)x - \Delta'/2
$$
\cite[Proposition 2.6]{K2}. On the other hand,
$$
-\Delta'/2 \,=\, \theta - (d+g-1)x
$$
\cite[Lemma 7]{K1}. The lemma follows from these two facts.
\end{proof}

Recall the formula of Macdonald \cite[\S~11]{Ma}
\begin{equation}\label{eq:kcd}
H^0(C_d,\, K_{C_d})\, =\, \bigwedge\nolimits^d H^0(C,\, K_C)
\end{equation}

\begin{lemma}
\label{usingconj}
If Conjecture \ref{conjecture} holds, then for any curve $C$
with $g\geq 3$ and $1< d < g$, the symmetric product
$C_d$ admits no non-trivial semi-orthogonal decomposition.
\end{lemma}

\begin{proof}
The class $\theta$ is nef, being the pullback of an ample class, while
the class $x$ is ample, hence $K_{C_d}\,=\, (g-d-1) x + \theta$ is nef
under the given conditions on $d$.
Furthermore $H^0(C_d,\, K_{C_d})\,=\,\bigwedge^d H^0(C,\, K_C)\,\neq\, 0$ \eqref{eq:kcd}, 
so all the conditions in Conjecture \ref{conjecture} are satisfied.
\end{proof}

\section{Base locus of canonical divisor of $C_d$}

Let $C$ be a smooth projective curve of genus $g$. Take any positive integer
$d\, \leq\, g-1$. In this section we prove that the base locus of the
canonical line bundle $K_{C_d}$ of the symmetric product $C_d$
coincides with the locus where the Albanese map is not injective
(Proposition \ref{baselocus}).
We give two independent proofs. The
first proof is a combination of Proposition \ref{prop1} and Lemma
\ref{lem1}. The 
second proof is given after the statement of Proposition \ref{baselocus}.

When we write a point of $C_d$ as
$z=(z_1,\, \cdots, \, z_d)$, the points $z_i$ of $C$ need not be distinct.
We also denote by $z$ the subscheme of $C$ defined by $\sum z_i$.

\begin{proposition}\label{prop1}
Let $1\,\leq \,d\,\leq\, g-1$. 
Let $z\,=\, (z_1,\, \cdots, \, z_d)\, \in\, C_d$ be a point of the base locus
of the complete linear system $|K_{C_d}|\,=\, \PP(H^0(C_d,\, K_{C_d}))$. Then the dimension of
$$
H^0(C,\, {\mathcal O}_C(z))\,=\, H^0(C,\, {\mathcal O}_C(\sum_{i=1}^d z_i))
$$
is at least two.
\end{proposition}

\begin{proof}
We shall first describe a subset of $\PP(H^0(C_d,\, K_{C_d}))$ whose linear span is whole
$\PP(H^0(C_d,\, K_{C_d}))$.

Let $S\, \subset\, H^0(C,\, K_{C})$ be a linear subspace of dimension $d$. Now define
$$
D_S \, :=\, \{(y_1,\, \cdots, \, y_d)\, \in\, C_d\, \mid\, {\rm div}(\omega) - \sum_{i=1}^d y_i
\ \text{~is effective for some~}\ \omega\,\in\, S\setminus\{0\}\}\, .
$$
Note that ${\rm div}(\omega)-\sum_{i=1}^dy_i$ is effective if and only if
$\omega$ vanishes on the subscheme of $C$ defined by $\sum_{i=1}^dy_i$.

We claim that $D_S$ is a divisor on $C_d$ linearly equivalent to $K_{C_d}$
 and moreover the collection
$\{D_S\}_{S\in \text{Gr}(d, H^0(C,\, K_C))}$ spans $H^0(C_d,\, K_{C_d})$.

To prove this, using \eqref{eq:kcd} the above divisor $D_S$ corresponds to the line
$$
\bigwedge\nolimits^d S\, \subset\, \bigwedge\nolimits^d H^0(C,\, K_C)\,=\, H^0(C_d,\, K_{C_d})\, .
$$
The collection of all such lines with $S$ running over $\text{Gr}(d, H^0(C,\, K_{C}))$
evidently spans $\bigwedge^d H^0(C,\, K_C)$, proving the claim
(for more details, see \cite[Lemma 1.5.4]{Ba}).

As in the statement of the proposition, take a point $z\,=\, (z_1,\, \cdots, \, z_d)\, \in\, 
C_d$ of the base locus of $\PP(H^0(C_d,\, K_{C_d}))$. Note that this means that for every linear 
subspace $S\, \subset\, H^0(C,\, K_{C})$ of dimension $d$, there is a nonzero $\omega\, 
\in\, S$ such that ${\rm div}(\omega) - \sum_{i=1}^d z_i$ is effective. We shall now interpret
this condition in order to be able to use it. Consider the short exact sequence
of sheaves
\begin{equation}\label{sh0}
0\, \longrightarrow\, K_C\otimes {\mathcal O}_C(-z)
\, \longrightarrow\, K_C \, \longrightarrow\, K_C\vert_z \, \longrightarrow\, 0
\end{equation}
on $C$. Let
\begin{equation}\label{ebeta}
0\, \longrightarrow\, H^0(C,\, K_C\otimes {\mathcal O}_C(-z))
\, \stackrel{\beta}{\longrightarrow}\, H^0(C,\, K_C) \, \stackrel{\gamma}{\longrightarrow}\,
H^0(K_C\vert_z)
\end{equation}
be the long exact sequence of cohomologies associated to \eqref{sh0}. This implies that
$$
\dim \beta(H^0(C,\, K_C\otimes {\mathcal O}_C(-z)))\, \geq\, g-d\, ,
$$
because $\dim H^0(K_C\vert_z)\,=\, d$.

We shall show that
\begin{equation}\label{aci}
\dim \beta(H^0(C,\, K_C\otimes {\mathcal O}_C(-z)))\, \geq\, g-d+1\, .
\end{equation}

To prove this, if $\dim \beta(H^0(C,\, K_C\otimes {\mathcal O}_C(-z)))\, =\, g-d$, then take a subspace
of dimension $d$
$$
S\, \subset\, H^0(C,\, K_C)
$$
which is complementary to the subspace
$\beta(H^0(C,\, K_C\otimes {\mathcal O}_C(-z)))$ of $H^0(C,\, K_C)$. Then the restriction
$\gamma\vert_S$,
where $\gamma$ is the homomorphism in \eqref{ebeta}, is injective. Therefore,
there is no nonzero $\omega\, \in\, S$ such that ${\rm div}(\omega) - \sum_{i=1}^d z_i$ is effective,
because such an element $\omega$ has to be in the kernel of $\gamma\vert_S$. This proves
\eqref{aci}.

From \eqref{aci} and \eqref{ebeta} it follows immediately that the homomorphism $\gamma$ in \eqref{ebeta} is 
\textit{not} surjective.

Now consider the short exact sequence of sheaves
$$
0\, \longrightarrow\, {\mathcal O}_C\, \longrightarrow\, {\mathcal O}_C(z)\, \longrightarrow\, {\mathcal O}_C(z)\vert_z\, \longrightarrow\, 0
$$
on $C$. Let
\begin{equation}\label{ep0}
0\, \longrightarrow\, H^0(C,\, {\mathcal O}_C) \, \longrightarrow\, H^0(C,\, {\mathcal O}_C(z))
\, \stackrel{\eta}{\longrightarrow}\, H^0({\mathcal O}_C(z)\vert_z) \, \stackrel{\phi}{\longrightarrow}\,
H^1(C,\, {\mathcal O}_C)
\end{equation}
be the corresponding long exact sequence of cohomologies. By Serre duality,
$$
H^1(C,\, {\mathcal O}_C)\,=\, H^0(C,\, K_C)^*\, .
$$
Using this duality, the homomorphism $\phi$ in \eqref{ep0} is the dual of the homomorphism
$\gamma$ in \eqref{ebeta}. We proved earlier that $\gamma$ is not surjective. Consequently,
$\phi$ is not injective. Hence from \eqref{ep0} it follows that
$\dim H^0(C,\, {\mathcal O}_C(z))\, \geq\, 2$. This completes the proof.
\end{proof}

In view of the above proof, the following converse of Proposition \ref{prop1} is now rather straightforward.

\begin{lemma}\label{lem1}
Let $1\,\leq \,d\,\leq\, g-1$. 
Let $z\,=\, (z_1,\, \cdots, \, z_d)\, \in\, C_d$ be a point such that the dimension of
$$
H^0(C,\, {\mathcal O}_C(z))\,=\, H^0(C,\, {\mathcal O}_C(\sum_{i=1}^d z_i))
$$
is at least two. Then $z$ lies on the base locus
of the complete linear system $|K_{C_d}|\,=\, \PP(H^0(C_d,\, K_{C_d}))$.
\end{lemma}

\begin{proof}
Since $\dim H^0(C,\, {\mathcal O}_C(z))\,\geq \, 2$, the homomorphism $\eta$ in \eqref{ep0}
is nonzero. Hence $\phi$ in \eqref{ep0} is not injective. Consequently, the dual homomorphism
$\gamma$ in \eqref{ebeta} is not surjective. Therefore,
$$
\dim \gamma (H^0(C,\, K_C)) \, <\, \dim H^0(K_C\vert_z)\,=\, d\, .
$$
This implies that for any linear subspace $S\, \subset\, H^0(C,\, K_C)$ of dimension $d$,
the restriction $\gamma\vert_S$ is not injective. Now for any nonzero
$\omega \, \in\, \text{kernel}(\gamma\vert_S)$ the divisor ${\rm div}(\omega) - \sum_{i=1}^d y_i$
is effective. Consequently, $z$ lies on the base locus
of the complete linear system $|K_{C_d}|$.
\end{proof}

Let $f\, :\, C\, \longrightarrow\, {\mathbb P}^1$ be a surjective map of degree $d$.
For any $b\, \in\, {\mathbb P}^1$, we have $f^{-1}(b)\, \in\, C_d$, where
$f^{-1}(b)$ is the scheme theoretic inverse image. Therefore, we have
morphism
$$
\widehat{f}\, :\, {\mathbb P}^1\, \longrightarrow\, C_d\, , \ \ b\, \longmapsto\, f^{-1}(b)\, .
$$
This is a morphism, and not just a set theoretic map, because $C_d$ is
the Hilbert scheme $\Hilb^d(C)$ of subschemes of $C$ of dimension $0$ and
length $d$, the graph of $f$ gives a closed subscheme of $\PP^1\times
C$, flat over $\PP^1$, and the morphism $\PP^1\too C_d$ associated to
this subscheme by the universal property of the Hilbert
scheme is precisely $\widehat{f}$.

\begin{corollary}\label{cor1}
The image of the above map $\widehat{f}$ is contained in the base locus
of the complete linear system $|K_{C_d}|$.
\end{corollary}

\begin{proof}
For any $b\, \in\, {\mathbb P}^1$, we have
$$
\dim H^0(C,\, {\mathcal O}_C(f^{-1}(b)))\, \geq\, \dim H^0({\mathbb P}^1,\,
{\mathcal O}_{{\mathbb P}^1}(b))\,=\, 2\, .
$$
So Lemma \ref{lem1} completes the proof.
\end{proof}

\begin{proposition}\label{baselocus}
Let $1\leq d\leq g-1$. The base locus of the canonical divisor $K_{C_d}$
is the set of points $(x_1,\cdots,x_d)$ in $C_d$ such that $h^0(\SO_C(x_1+\cdots+x_d))>1$.

Equivalently, the base locus is the set of points in $C_d$ where the Albanese map is not injective. 
\end{proposition}

\begin{proof}
The first part follows from the combination of Proposition \ref{prop1} and Lemma \ref{lem1}. The
second part follows from the observation that the fiber of the Albanese
map $u:C_d\too J(C)$ is $u^{-1}(u(z))=\PP(H^0(\SO_C(z)))$.

\end{proof}

We shall now give the second proof of Proposition \ref{baselocus}. Consider the Albanese map
$$
u:C_d\too J(C)
$$
Let $z\,=\,(z_1,\cdots,z_d)\,\in\, C_d$ be a point, which can also be
thought as a subscheme in $C$. The tangent space $T_z C_d$ of $C_d$ at $z$ is
\begin{equation}\label{tangent}
T_zC_d\,=\,\Hom(\SO_C(-z),\,\SO_z)\,=\,H^0(C,\,\SO_C(z)|_z)\, .
\end{equation}
Therefore, the differential of the Albanese map gives a linear map
$$
\phi\,:\,H^0(C,\SO_C(z)|_z)\,=\, T_zC_d \,\stackrel{du_z}\too\, T_{u(z)}J(C)=H^1(C,\,\SO_C)\, .
$$
This map $\phi$ is the connecting homomorphism
in the
long exact sequence of cohomologies associated to the short exact sequence of sheaves
\begin{equation}
\label{seq1}
0 \too \SO_C \too \SO_C(z) \too \SO_C(z)|_z \too 0
\end{equation}
on $C$.
The dual to the map $\phi$ is the homomorphism
$$
\gamma\,:\,H^0(C,K_C) \,\too\, H^0(C,K_C|_z)
$$
in the long exact sequence of
cohomologies associated to the short exact sequence
$$
0 \too K_C(-z) \too K_C \too K_C|_z \too 0
$$
obtained by applying the functor $Hom(\cdot,K_C)$
to the short exact sequence \eqref{seq1}.

Using \eqref{tangent} and Serre duality,
we obtain the following canonical isomorphism for the fiber
of $K_{C_d}$ over a point $z=(z_1,\cdots,z_d)\in C_d$
\begin{equation}\label{eq:fiberkcd}
K_{C_d}|_z\, = \, \bigwedge\nolimits^d \Hom(\SO_C(-z),\SO_z)^\vee
 \, = \, \bigwedge\nolimits^d \Ext^1(\SO_z,K_C(-z))
 \, = \, \bigwedge\nolimits^d H^0(K_C|_z)
\end{equation}

Using the identifications \eqref{eq:kcd}, and \eqref{eq:fiberkcd}
and taking the $d$-fold exterior product, we get that
$$
e_z:H^0(C_d,\, K_{C_d})\, =\, \bigwedge\nolimits^d H^0(C,\, K_C)
\stackrel{\wedge^d \gamma}{\too} \bigwedge\nolimits^d H^0(K_C|_z)=K_{C_d}|_z\, .
$$
The above map $e_z$ is the evaluation map at
$z$. We note that a point $z\in C_d$ is in the base locus of $K_{C_d}$ if and only
if $e_z$ is zero. This is equivalent to the map $\gamma$ being non-surjective,
which in turn is equivalent to the assertion that the map $\phi$ is not injective.
We have identified the map $\phi$ with the differential $du_z$ of the
Albanese map at $z$. Therefore, a point $z\in C_d$ is in the base locus of
$K_{C_d}$ if and only the Albanese map is not injective at $z$.

\section{Semi-orthogonal decompositions of 
$D(C_d)$}

\begin{definition}
\label{defsod}
A triangulated category $\mathcal{T}$ admits a nontrivial semi-orthogonal decomposition if 
there are two full non-trivial triangulated subcategories $\mathcal{A}, \mathcal{B}$ of $\mathcal{T}$ such that 
\begin{enumerate} 
\item $\mathrm{Hom}_{\mathcal{T}}(b,a)=0$ for every $b \in \mathcal{B}$, $a \in \mathcal{A}$ and 
\item $\mathcal{A}, \mathcal{B}$ generate $\mathcal{T}$. 
\end{enumerate} 
\end{definition}

In this section, using our results on the base locus of the canonical
bundle, and applying the work of Kawatani and Okawa \cite{KO}, we will obtain
restrictions to the existence of semi-orthogonal decompositions of the
triangulated category $D(C_d)$.

\begin{theorem}[{\cite[Corollary 1.3]{KO}}]
\label{KO-finite}
Let $X$ be a smooth projective variety such that 
the base locus of the canonical divisor is a finite set. Then 
there is no non-trivial semi-orthogonal decomposition of $D(X)$.
\end{theorem}

\begin{corollary}\label{cor:cd-sod}
Let $C$ be a smooth complex projective curve of genus $g\geq 3$ and
let $d$ be a integer with $d<\gon(C)$. Then there is no non-trivial
semi-orthogonal decomposition of $D(C_d)$.
\end{corollary}

\begin{proof}
Note that $h^0(\SO_C(\sum z_i))=1$ for any $z=(z_1,\cdots,z_d)$,
because $d<\gon(C)$. So Proposition \ref{prop1} implies that $z$ is not a base
point of $K_{C_d}$, and hence the canonical divisor $K_{C_d}$ is base-point
free. Now from Theorem \ref{KO-finite} it follows that $D(C_d)$ has no
non-trivial semi-orthogonal decomposition. 
\end{proof}

When $d\,=\,2$ we are able to prove a stronger result which disposes
of condition on the gonality of $C$. We will use the following result:

\begin{theorem}[{\cite[Theorem 1.8]{KO}}]
\label{KO-surface}
Let $S$ be a minimal smooth projective surface of general type 
with $h^0(K_S)>1$ and satisfying the condition that for any one-dimensional connected
component $Z\subset \op{Bs}|K_S|$, its intersection matrix is negative
definite. 
Then 
there is no non-trivial semi-orthogonal decomposition of $D(S)$.
\end{theorem}

We start with some results about the geometry of $C_2$.

\begin{lemma}
\label{lemma1}
Let $g\geq 3$. The surface $C_2$ is minimal. It has an embedded rational curve if and only
if $C$ is hyperelliptic, and in this case
\begin{itemize}
\item the rational curve is
$\Gamma=\{x+\sigma(x)\}$, where $\sigma$ is the hyperelliptic involution, and 

\item $\Gamma^2=1-g$, i.e., $\Gamma$ is a $(1-g)$-curve.
\end{itemize}
\end{lemma}

\begin{proof}
For all points in the image of the Albanese map $u:C_2 \too J(C)$, the fiber is a projective space \eqref{fiberalbanese}. Since $C_2$ is a surface, the fiber of the Albanese has at most dimension 2. But the fiber cannot be $\PP^2$, because this would mean that there is a degree 2 line bundle $A$ on $C$ with $h^0(A)=3$, and this would imply that the genus of $C$ is $\PP^1$, contradicting the hypothesis $g\,\geq\,3$. In particular, the Albanese map is not constant.

If $C$ is hyperelliptic then $C_2$ has no rational curve. Indeed, a rational curve has to map to a point, so the fiber over this point would be exactly $\PP^1$ (by the previous argument), and we would have a line bundle $A$ on $C$ with $h^0(A)=2$, so $A$ would be the hyperelliptic divisor.

Let us now suppose that $C$ is hyperelliptic. A rational curve in
$C_2$ has to be in a fiber of the Albanese map,
but the only positive dimensional fiber of this map is one dimensional
(by the argument in the first paragraph of this proof), 
and it is the fiber over
the hyperelliptic line bundle.

We denote the above mentioned fiber of $u$ by $\Gamma$, so $\Gamma$
is isomorphic to $\PP^1$. We now calculate its self-intersection.
The self-intersection of the diagonal $\Delta\subset
C\times C$ is
$$\Delta^2=2-2g$$
by Poincar\'e--Hopf theorem. 
The automorphism of $C \times C$, which is identity on the first factor
and the hyperelliptic involution on the second, sends the diagonal 
$\Delta$ to the graph $\wt\Gamma=\{x,\sigma(x)\}$ of the hyperelliptic
involution, hence also $\wt\Gamma^2=2-2g$. Consider the diagram
$$
\xymatrix{
{\wt\Gamma} \ar[d]_{f} \ar@{^(->}[r] & {C\times C} \ar[d]^{\pi} \\ 
{\Gamma} \ar[r] & {C_2} \\ 
}
$$
Observe that $\wt\Gamma\cong C$ and the morphism $f$ is just the
quotient by the hyperelliptic involution.
Therefore we have $\pi_*\wt\Gamma = 2\Gamma$ and
$\pi^*\Gamma = \wt\Gamma$ as cycles, and the projection formula for
intersection gives
$$
\pi_*(\wt\Gamma \cdot \wt\Gamma)=\pi_*(\wt\Gamma \cdot \pi^*\Gamma)=
\pi_*(\wt\Gamma) \cdot \Gamma=2\Gamma\cdot \Gamma
$$
and hence $\Gamma^2=1-g$.
\end{proof}
 
Now we prove that $C_2$ is a surface of general type. 
 
\begin{lemma}
\label{lemma2}
If $g \geq 3$, then the symmetric product $C_2$ is of general type.
\end{lemma}

\begin{proof}
In view of \cite[Proposition X.1]{Be} it suffices to show that
\begin{itemize}
\item the self-intersection of the canonical divisor of $C_2$ is positive, and
$C_2$ is irrational surface.
\end{itemize}

{}From \cite[Proposition I.8]{Be}, we can compute the self-intersection on $C^2.$ We know that
the pull back of the canonical divisor on $C_2$ to $C^2$ is $K_C \boxtimes K_C ( -\Delta)$. The
self-intersection of $K_C \boxtimes K_C ( -\Delta)$ is
$$
2(2g-2)^2-(2g-2)-4(2g-2)=(2g-2)(4g-9)
$$
and it is positive when $g \geq 3$. 

To prove that $C_2$ is not rational by contradiction, assume that $C_2$ is rational. Then $C_2$ can be covered by 
rational curves, which implies that the Albanese map $u:C_2\to J(C)$
is constant, but we know that the Albanese map is not constant
by the argument in the first paragraph of the proof of \ref{lemma1}.
Hence $C_2$ is not rational.

Therefore $C_2$ is of general type when $g \geq 3$.
\end{proof}

\begin{remark}
When $g\,=\,2$, we know that $C_2$ is the blow-up of the $J(C)$ at a point. This has two consequences: $C_2$ is not a surface of general type, and $D(C_2)$ admits a nontrivial semi-orthogonal decomposition (using the blow-up formula in \cite{Or}).
\end{remark}

Finally we check that $p_g > 1$ for $C_2$.

\begin{lemma}
\label{lemma3}
The canonical bundle $K_{C_2}$ has $h^0(K_{C_2})={g \choose 2}>1$
(recall that $g\geq 3$).
\end{lemma}

\begin{proof}
Macdonald, \cite{Ma}, proves that $H^0(C_d,\, K_{C_d})\,=\,\bigwedge^d H^0(C,\, K_C)$.
\end{proof}

\begin{theorem}\label{c2-sod}
Let $C$ be a smooth projective curve of genus $g \geq 3$.
Then there is no non-trivial semi-orthogonal decomposition on $D(C_2)$.
\end{theorem}

\begin{proof}
If $C$ is hyperelliptic, by Lemma \ref{lemma1} the Albanese map fails to
be injective exactly on $\Gamma\subset J(C)$. Therefore,
Proposition \ref{baselocus} implies that
$$
\op{Bs}|K_{C_2}|=\Gamma
$$
and the only connected component of the base locus is $\Gamma$, 
which is irreducible. Hence the intersection matrix is 
just $\Gamma^2=1-g<0$, so it is negative definite.

If $C$ is not hyperelliptic, then the Albanese map is injective. Proposition 
\ref{baselocus} implies that $K_{C_d}$ has no base locus, and hence there is no non-trivial semi-orthogonal decomposition by Theorem 
\ref{KO-finite}.

In view of Lemmas \ref{lemma1}, \ref{lemma2}, and \ref{lemma3},
the hypothesis of Theorem \ref{KO-surface} are satisfied, so there is
no non-trivial semi-orthogonal decomposition.
\end{proof}

\begin{remark}
We thank an anonymous referee for the following alternative argument 
to check the assumptions of Theorem \ref{KO-surface}. By Lemma
\ref{lemma:canonicald} and the assumption $g\,\geq\,3$, we know that
the canonical divisor $K_{C_2}$ is nef and big. Hence $C_2$ is minimal
surface of general type. Since $C_2$ is a minimal model, the Albanese
map should be birational onto its image (otherwise, $C_2$ is either
$\PP^2$ or a ruled surface, contradiction). In other words, the
morphism of $C_2$ onto its image is a resolution of singularities of a
surface $u(C_2)$. Now the base locus of $K_{C_2}$ coincides with the
$u$-exceptional curve (Proposition \ref{baselocus}), hence the
intersection matrix of each connected component is negative definite.
\end{remark}

\section*{Acknowledgements}

It is a pleasure to express our deep gratitude to M. S. Narasimhan
for drawing our attention to this problem. T. G. and K.-S. L. thank
him for many helpful discussions which took place during several visits to the Indian Institute of Science (Bangalore). 
We heartily thank Gadadhar Misra for kind hospitality.
I. B. thanks C. Ciliberto for a very useful correspondence.
We thank R. Thomas for informing us about \cite{To} and
S. Mukhopadhyay for \cite{BGM}.

This article was finished during a visit of I.B. and T.G. to the
International Center for Theoretical Sciences (ICTS) during the
program Quantum Fields, Geometry and Representation Theory (Code:
ICTS/qftgrt/2018/07).
We thank S. Mehrotra for discussions and K.-S. Lee thanks S. Okawa
for answering questions.
Part of this work was done while K.-S. Lee
was a research fellow of KIAS.

I. B. is partially supported by a J. C. Bose Fellowship.
T. G. is supported by
the European Union (61253 MODULI 7th Framework Programme), 
the Spanish Ministerio de Ciencia e Innovaci\'on
(MTM2016-79400-P and Severo Ochoa programme for Centres of Excelence in R\&D SEV-2015-0554), and CSIC  (\textit{Ayuda extraordinaria a Centros de Excelencia Severo Ochoa} 20205CEX001). 
K.-S. L. is supported by
the Institute for Basic Science (IBS-R003-Y1).


\begin{thebibliography}{AAAA}

\bibitem[AB]{AB}{A. Auel and M. Bernardara, }
{\it Cycles, derived categories, and rationality}, 
in Surveys on Recent Developments in Algebraic Geometry, Proceedings
of Symposia in Pure Mathematics \textbf{95} (2017), 199--266.

\bibitem[ACGH]{ACGH}{E. Arbarello, M. Cornalba, P. Griffiths and J. Harris, }
{\it Algebraic Curves}, Grundlehren der mathematischen Wissenschaften 267.
Springer-Verlag 1985.


\bibitem[Ba]{Ba}{F. Bastianelli, }
{\it  The geometry of second symmetric products of curves,}
(Ph.D. Thesis, Universit\`a degli Studi di Pavia, 2009).

\bibitem[Be]{Be} {A. Beauville, } \textit{Complex algebraic surfaces,} Translated from the French by R. Barlow, N. 
I. Shepherd-Barron and M. Reid. London Mathematical Society Lecture Note Series, 68. Cambridge University Press, 
Cambridge, 1983.

\bibitem[BB]{BB}{M. Bernardara and M. Bolognesi, }
Derived categories and rationality of conic bundles, 
\textit{Comp. Math.} \textbf{149} (2013), 1789--1817.

\bibitem[BGM]{BGM} {P. Belmans, S. Galkin and S. Mukhopadhyay, }
  \textit{Semiorthogonal decompositions for moduli of sheaves on curves, }
  Oberwolfach Report No. 24/2018, 9--11, DOI:10.4171/OWR/2018/24

\bibitem[FK]{FK} {A. Fonarev and A. Kuznetsov, }
Derived categories of curves as components of Fano manifolds, 
\textit{Jour. Lond. Math. Soc.} \textbf{97} (2018), 24--46.

\bibitem[Ha]{Ha}{R. Hartshorne, }
\textit{Algebraic Geometry}, Graduate Texts in Mathematics, No. 52.
Springer-Verlag, New York-Heidelberg, 1977.

\bibitem[KO]{KO}{K. Kawatani and S. Okawa, }
Nonexistence of semiorthogonal decompositions and sections
of the canonical bundle,
\textit{Preprint 2015}
\texttt{arxiv:1508.00682}. 

\bibitem[K1]{K1}{A. Kouvidakis, }
Divisors on symmetric products of curves,
{\it Trans. Amer. Math. Soc.} \textbf{337} (1993), 117--128.

\bibitem[K2]{K2}{A. Kouvidakis, }
On some results of Morita and their application to questions of ampleness,
{\it Math. Zeit.} \textbf{241} (2002), 17--33.

\bibitem[Ku10]{Ku10} {A. Kuznetsov, }
\textit{Derived categories of cubic fourfolds,} 
in Cohomological and geometrical approaches to rationality problems, 
Progress in Mathematics, vol. 282 (Birkh\"auser, Boston,
MA, 2010), 163--208.

\bibitem[Ku12]{Ku12}{A. Kuznetsov, }
Instanton bundles on Fano threefolds,
\textit{Cent. Eur. J. Math.} \textbf{10} (2012), 1198--1231.

\bibitem[Ku14]{Ku14}{A. Kuznetsov, }
\textit{Semiorthogonal decompositions in algebraic geometry.} 
Proceedings of the International Congress of Mathematicians-Seoul 2014. Vol. II, 635-660, Kyung Moon Sa, Seoul, 2014.

\bibitem[Ma]{Ma}{I. Macdonald, }
Symmetric products of an algebraic curve,
\textit{Topology} \textbf{1} (1962), 319--343.

\bibitem[Lee]{Lee}{K.-S. Lee, } 
Remarks on motives of moduli spaces of rank 2 vector bundles on curves, 
\textit{Preprint 2018,}
\texttt{arXiv:1806.11101}.

\bibitem[Na1]{Na1}{M. S. Narasimhan, } 
Derived categories of moduli spaces of vector bundles on curves, 
\textit{Jour. Geom. Phy.} \textbf{122} (2017), 53--58.

\bibitem[Na2]{Na2}{M. S. Narasimhan, }
Derived categories of moduli spaces of vector bundles on curves II,
in  \textit{Geometry, Algebra, Number Theory, and Their Information Technology Applications,} Toronto, Canada, June, 2016, and Kozhikode, India, August, 2016.
 Springer Proceedings in Mathematics and Statistics, \textbf{251} 2018, 375--382.

\bibitem[Ok]{Ok}{S. Okawa, }
Semiorthogonal decomposability of the derived category of a curve,
\textit{Adv. Math.} \textbf{228}, (2011), 2869--2873.

\bibitem[Or]{Or}{D. Orlov, }
Projective bundles, monoidal transformations, and derived categories
of coherent sheaves,
\textit{Izv. Akad. Nauk SSSR Ser. Mat.,} \textbf{56} (1992) 852--862;
English transl., \textit{Russian Acad. Sci. Izv. Math.,} \textbf{41}
(1993) 133--141.

\bibitem[To]{To}{Y. Toda, }
  Semiorthogonal decompositions of stable pair moduli spaces via
  $d$-critical flips,
\textit{Preprint 2018,}
\texttt{arXiv:1805.00183v1}

\end{thebibliography}
\end{document}